\documentclass{amsart}
\usepackage{amssymb}
\usepackage{graphicx}

\newtheorem{thm}{Theorem}
\newtheorem{cor}{Corollary}
\newtheorem{lem}{Lemma}

\newtheorem{rem}{Remark}

\newcommand{\A}{{\mathcal A}}

\newcommand{\U}{{\mathcal U}}
\newcommand{\es}{{\mathcal S}}

\newcommand{\D}{{\mathbb D}}

\def\be{\begin{equation}}
\def\ee{\end{equation}}

\newcommand{\bee}{\begin{enumerate}}
\newcommand{\eee}{\end{enumerate}}

\newcommand{\blem}{\begin{lem}}
\newcommand{\elem}{\end{lem}}
\newcommand{\bthm}{\begin{thm}}
\newcommand{\ethm}{\end{thm}}
\newcommand{\bcor}{\begin{cor}}
\newcommand{\ecor}{\end{cor}}
\newcommand{\beg}{\begin{example}}
\newcommand{\eeg}{\end{example}}
\newcommand{\begs}{\begin{examples}}
\newcommand{\eegs}{\end{examples}}
\newcommand{\bdefe}{\begin{defin}}
\newcommand{\edefe}{\end{defin}}
\newcommand{\bprob}{\begin{prob}}
\newcommand{\eprob}{\end{prob}}
\newcommand{\bei}{\begin{itemize}}
\newcommand{\eei}{\end{itemize}}

\newcommand{\bcon}{\begin{conj}}
\newcommand{\econ}{\end{conj}}
\newcommand{\bcons}{\begin{conjs}}
\newcommand{\econs}{\end{conjs}}
\newcommand{\bprop}{\begin{propo}}
\newcommand{\eprop}{\end{propo}}
\newcommand{\br}{\begin{rem}}
\newcommand{\er}{\end{rem}}
\newcommand{\brs}{\begin{rems}}
\newcommand{\ers}{\end{rems}}
\newcommand{\bo}{\begin{obser}}
\newcommand{\eo}{\end{obser}}
\newcommand{\bos}{\begin{obsers}}
\newcommand{\eos}{\end{obsers}}
\newcommand{\bpf}{\begin{pf}}
\newcommand{\epf}{\end{pf}}
\newcommand{\ba}{\begin{array}}
\newcommand{\ea}{\end{array}}
\newcommand{\beq}{\begin{eqnarray}}
\newcommand{\beqq}{\begin{eqnarray*}}
\newcommand{\eeq}{\end{eqnarray}}
\newcommand{\eeqq}{\end{eqnarray*}}

\begin{document}
\bibliographystyle{amsplain}

\title[Application of Grunsky coefficients in univalent functions]{Some application of Grunsky coefficients in the theory of univalent functions}

\author[M. Obradovi\'{c}]{Milutin Obradovi\'{c}}
\address{Department of Mathematics,
Faculty of Civil Engineering, University of Belgrade,
Bulevar Kralja Aleksandra 73, 11000, Belgrade, Serbia}
\email{obrad@grf.bg.ac.rs}

\author[N. Tuneski]{Nikola Tuneski}
\address{Department of Mathematics and Informatics, Faculty of Mechanical Engineering, Ss. Cyril and Methodius
University in Skopje, Karpo\v{s} II b.b., 1000 Skopje, Republic of North Macedonia.}
\email{nikola.tuneski@mf.edu.mk}

\subjclass[2020]{30C45, 30C50, 30C55}
\keywords{univalent functions, Grunsky coefficients, third logarithmic coefficient, coefficient difference, generalised Zalcman conjecture, second Hankel determinant, third Hankel determinant.}




\begin{abstract}
Let function  $f$ be  normalized, analytic and univalent in the unit disk $\D=\{z:|z|<1\}$ and  $f(z)=z+\sum_{n=2}^{\infty} a_n z^n$.
Using a method based on Grusky coefficients we study several problems over that class of univalent functions: upper bounds of the special case of the generalised Zalcman conjecture $|a_2a_3-a_4|$, of the third logarithmic coefficient, and of the second Hankel determinant for the logarithmic coefficients.
\end{abstract}

\maketitle

\medskip

\section{Introduction and definitions}

\medskip

Let $\mathcal{A}$ be the class of functions $f$ which are analytic  in the open unit disc $\D=\{z:|z|<1\}$ of the form
\be\label{e1}
f(z)=z+a_2z^2+a_3z^3+\cdots,
\ee
and let $\mathcal{S}$ be the subclass of $\mathcal{A}$ consisting of functions that are univalent in $\D$.

\medskip

Although the famous Bieberbach conjecture $|a_n|\le n$ for $n\ge2,$ was proved by de Branges  in 1985 \cite{Bra85}, a great many other problems concerning the coefficients $a_n$ remain open.

\medskip
One of them is the generalized Zalcman conjecture
$$|a_n a_m-a_{n+m-1}|\le(n-1)(m-1),$$ $n\ge2$, $m\ge2$, closed by Ma  for the class of starlike functions and for the class of univalent functions with real coefficients and by Ravichandran and Verma in \cite{ravi} for the classes of starlike and convex functions of given order and for the class of functions with bounded turning. In \cite{OB-NT-n2} the authors studied the generalized Zalcman conjecture for the class
\[\U=\left\{f\in\A:\left| \left(\frac{z}{f(z)}\right)^2 f'(z) -1\right|<1, \, z\in\D \right\}\]
and proved it for the cases $m=2$, $n=3$; and $m=2$, $n=4$. In this paper we prove the estimate $2.10064\ldots$ for the general class when $m=2$ and $n=3$ which is close to the conjectured value 2.

\medskip

Another, still open problem, is finding sharp estimates of logarithmic coefficient, $\gamma_n$, of a univalent function $f(z)=z+a_2z^2+a_3z^3+\cdots,$  defined by
  \be\label{log-co-1} F_f(z) := \log\frac{f(z)}{z}=2\sum_{n=1}^\infty \gamma_n z^n.\ee
Relatively little exact information is known about the coefficients. The natural conjecture $|\gamma_n|\le1/n$, inspired by the Koebe function (whose logarithmic coefficients are $1/n$) is false even in order of magnitude (see Duren \cite[Section 8.1]{duren}).
For the class $\es$ the sharp estimates of single logarithmic coefficients S are known only for $\gamma_1$ and $\gamma_2$, namely,
\[|\gamma_1|\le1\quad\mbox{and}\quad |\gamma_2|\le \frac12+\frac1e=0.635\ldots,\]
and are unknown for $n\ge3$. In this paper we give the estimate $|\gamma_3|\le0.5566178\ldots$ for the general class of univalent functions. This is an improvement of $|\gamma_3|\le0.7688\ldots$ obtained in \cite{OT-Turk}.
For the subclasses of univalent functions  the situation is not a great deal better. Only the estimates of the initial logarithmic coefficients are available. For details see \cite{cho}.

\medskip

The upper bound of the Hankel determinant is a problem rediscovered and extensively studied in recent years. Over the class $\A$ of functions $f(z)=z+a_2z^2+a_3z^3+\cdots$ analytic on the unit disk, this determinant is defined by
\[
        H_{q,n}(f) = \left |
        \begin{array}{cccc}
        a_{n} & a_{n+1}& \ldots& a_{n+q-1}\\
        a_{n+1}&a_{n+2}& \ldots& a_{n+q}\\
        \vdots&\vdots&~&\vdots \\
        a_{n+q-1}& a_{n+q}&\ldots&a_{n+2q-2}\\
        \end{array}
        \right |,
\]
where $q\geq 1$ and $n\geq 1$. The second order Hankel determinants is
\[
 H_{2,2}(f) =  \left |
        \begin{array}{cc}
        a_2 & a_3\\
        a_3 & a_4\\
        \end{array}
        \right | = a_2a_4-a_{3}^2,
\]
and for the logarithmic coefficients:
\begin{equation}\label{loghank}
H_{2,1}(F_f/2) = \gamma_1\gamma_3-\gamma_2^2 = \frac14\left( a_2a_4-a_3^2 + \frac{1}{12}a_2^4\right) .
\end{equation}
For the general class $\es$ of univalent functions in the class $\A$ tehre are very few results concerning the Hankel determinant.  The best known for the second order case  is due to Hayman (\cite{hayman-68}), saying  that $|H_2(n)|\le An^{1/2}$, where $A$ is an absolute constant, and that this rate of growth is the best possible.
There are much more results for the subclasses of $\es$ and some references  are \cite{ janteng-06,janteng-07, DTV-book}. Much less is known about the bounds of the modulus of the Hankel determinant for the logarithmic coefficients. In \cite{Kowalczyk-22,Kowalczyk-23} the authors considered the cases of starlike, convex, strongly starlike and strongly convex functions and found the best possible results. In this paper we give estimate for the second order Hankel determinant for the general class of univalent functions.

\medskip

For the study of the probles defined above  we will use method based on Grunsky coefficients.
In the proofs we will use mainly the notations and results given in the book of N. A. Lebedev (\cite{Lebedev}).

\medskip

Here are basic definitions and results.

\medskip

Let $f \in \mathcal{S}$ and let
\[
\log\frac{f(t)-f(z)}{t-z}=\sum_{p,q=0}^{\infty}\omega_{p,q}t^{p}z^{q},
\]
where $\omega_{p,q}$ are called Grunsky's coefficients with property $\omega_{p,q}=\omega_{q,p}$.
For those coefficients we have the next Grunsky's inequality (\cite{duren,Lebedev}):
\be\label{eq 4}
\sum_{q=1}^{\infty}q \left|\sum_{p=1}^{\infty}\omega_{p,q}x_{p}\right|^{2}\leq \sum_{p=1}^{\infty}\frac{|x_{p}|^{2}}{p},
\ee
where $x_{p}$ are arbitrary complex numbers such that last series converges.

\medskip

Further, it is well-known that if $f$ given by \eqref{e1}
belongs to $\mathcal{S}$, then also
\be\label{eq 5-mo-4}
f_{2}(z)=\sqrt{f(z^{2})}=z +c_{3}z^3+c_{5}z^{5}+\cdots
\ee
belongs to the class $\mathcal{S}$. Then for the function $f_{2}$ we have the appropriate Grunsky's
coefficients of the form $\omega_{2p-1,2q-1}$ and the inequality \eqref{eq 4} has the form:
\be\label{eq 6-mo-5}
\sum_{q=1}^{\infty}(2q-1) \left|\sum_{p=1}^{\infty}\omega_{2p-1,2q-1}x_{2p-1}\right|^{2}\leq \sum_{p=1}^{\infty}\frac{|x_{2p-1}|^{2}}{2p-1}.
\ee

\medskip

Here, and futher in the paper we omit the upper index (2) in  $\omega_{2p-1,2q-1}^{(2)}$ if compared with Lebedev's notation.

\medskip

From inequality \eqref{eq 6-mo-5}, when $x_{2p-1}=0$ and $p=2,3,\ldots$, we have
\begin{equation}\label{e7}
|\omega_{11} x_1 +\omega_{31} x_3 |^2 +3|\omega_{13} x_1 +\omega_{33} x_3 |^2 + 5|\omega_{15} x_1 +\omega_{35} x_3 |^2 \le |x_1|^2+\frac{|x_3|^2}{3}.
\end{equation}

\medskip

As it has been shown in \cite[p.57]{Lebedev}, if $f$ is given by \eqref{e1} then the coefficients $a_{2}$, $ a_{3}$, $ a_{4}$ and $a_5$ are expressed by Grunsky's coefficients  $\omega_{2p-1,2q-1}$ of the function $f_{2}$ given by
\eqref{eq 5-mo-4} in the following way:
\be\label{e9}
\begin{split}
a_{2}&=2\omega _{11},\\
a_{3}&=2\omega_{13}+3\omega_{11}^{2}, \\
a_{4}&=2\omega_{33}+8\omega_{11}\omega_{13}+\frac{10}{3}\omega_{11}^{3},\\
0&= 3\omega_{15}-3\omega_{11}\omega_{13}+\omega_{11}^3-3\omega_{33}.
\end{split}
\ee

\medskip

\section{Generalized Zalcman conjecture}

In this section we consider the generalized Zalcman conjecture in the case $n=2$ and $m=3$.

\begin{thm}
If $f\in\es$ is given by \eqref{e1}, then
\[|a_2a_3-a_4| \le 2.10064\ldots.\]
\end{thm}

\begin{proof}
Using \eqref{e9} we have
\[|a_2a_3-a_4| = \left|2\omega_{33}+4\omega_{11}\omega_{13}-\frac83\omega_{11}^3\right|,\]
and further by \eqref{e11n}
\be\label{e16n}
\begin{split}
|a_2a_3-a_4|
&= \left|2\omega_{15}+2\omega_{11}\omega_{13}-2\omega_{11}^3\right|\\
&= \left|2\omega_{15}+(2\omega_{13}-\omega_{11}^2)\omega_{11}-\omega_{11}^3\right|\\
&\le 2|\omega_{15}|+|2\omega_{13}-\omega_{11}^2||\omega_{11}|+|\omega_{11}|^3.
\end{split}
\ee
Since $|2\omega_{13}-\omega_{11}^2|=|a_3-a_2^2|\le1$ (see \cite[p.5]{DTV-book}) for the class $\es$ and using \eqref{e14n}, from \eqref{e16n} we obtain
\[ |a_2a_3-a_4| \le x+x^3+\frac{2}{\sqrt5}\sqrt{1-x^2-3y^3} \equiv f_3(x,y),\]
where we put $|\omega_{11}|=x$, $|\omega_{13}|=y$ and $(x,y)\in \{(x,y):0\le x\le1, 0\le y \le \frac{1}{\sqrt3} \sqrt{1-x^2}\}\equiv E_3$. Since $\frac{\partial f_3}{\partial x} = \frac{-6y}{\sqrt5\sqrt{1-x^2-3y^2}}= 0$ if, and only if, $y=0$, we realize that  $f_3$ has no singular points inside $E_3$. On the boundary we have
\begin{itemize}
  \item[-] $f_3(x,0)=x+x^3+\frac{2}{\sqrt5}\sqrt{1-x^2}\le 2.10064\ldots$ for $0\le x\le1$, obtained for $x=0.9740\ldots$;
  \item[-] $f_3(0,y)=\frac{2}{\sqrt5}\sqrt{1-3y^2}\le \frac{2}{\sqrt5}<1$ for $0\le y\le1/\sqrt3$;
  \item[-] $f_3(1,y)=f_3(1,0)=2$;
  \item[-] $f_3\left(x,\frac{1}{\sqrt3}\sqrt{1-x^2}\right) = x+x^3\le 2$ for $0\le x\le1$.
\end{itemize}

\medskip

Finally
\[|a_2a_3-a_4| \le 2.10064\ldots.\]
\end{proof}

\medskip

\begin{rem}
We believe that $|a_2a_3-a_4| \le 2$ is true for the class $\es$.
\end{rem}

\medskip

\section{The third logarithmic coefficient}

\medskip

We now give upper bound of the third logarithmic coefficient over the class $\es$.

\medskip

\begin{thm}
Let $f\in\es$ and be given by \eqref{e1}. Then
\[|\gamma_3| \le 0.5566178\ldots.\]
\end{thm}

\medskip

\begin{proof}
From \eqref{log-co-1}, after differentiation and comparation of coefficients we receive
\[\gamma_3 =\frac12\left(a_4-a_2a_3+\frac13a_2^3\right).\]
The fifth relation in \eqref{e9} gives
\be\label{e11n}\omega_{33} =  \omega_{15}-\omega_{11}\omega_{13}+\frac13\omega_{11}^3,  \ee
which, together with the other expressions from \eqref{e9} implies
\[\gamma_3=\omega_{33}+2\omega_{11}\omega_{13} = \omega_{15}+\omega_{11}\omega_{13}+\frac13\omega_{11}^3.\]
Therefore,
\be\label{e12n}
|\gamma_3| \le \frac13|\omega_{11}|^3+|\omega_{11}||\omega_{13}|+|\omega_{15}|.
\ee
Now, choosing $x_1=1$ and $x_3=0$ in \eqref{e7} we have
\[ |\omega_{11} |^2 +3|\omega_{13} |^2 + 5|\omega_{15}|^2 \le 1, \]
and also from here
\[ |\omega_{11} |^2 +3|\omega_{13} |^2  \le 1. \]
The last two relations imply
\be\label{e14n}
|\omega_{13}|\le\frac{1}{\sqrt{3}}\sqrt{1-|\omega_{11}|^2} \quad \mbox{and}\quad |\omega_{15}|\le\frac{1}{\sqrt{5}}\sqrt{1-|\omega_{11}|^2-3|\omega_{13}|^2}.
\ee
Using \eqref{e12n} and \eqref{e14n} we have
\[
|\gamma_3| \le \frac13|\omega_{11}|^3+|\omega_{11}||\omega_{13}|+\frac{1}{\sqrt{5}}\sqrt{1-|\omega_{11}|^2-3|\omega_{13}|^2}\equiv f_1(|\omega_{11}|,|\omega_{13}|),
\]
where
\[ f_1(x,y) = \frac13x^3+xy+\frac{1}{\sqrt{5}}\sqrt{1-x^2-3y^2}\]
and $0\le x\le1$, $0\le y \le \frac{1}{\sqrt3} \sqrt{1-x^2}$ ($|a_2|=|2\omega_{11}|\le2$ implies $0\le|\omega_{11}|\le1$).

\medskip

So, we need to find maximum of the function $f_1$ over the region $E_1=\Big\{ (x,y):0\le x\le1, 0\le y \le \frac{1}{\sqrt3} \sqrt{1-x^2} \Big\}$.

\medskip

The system
\[
\begin{cases}
\frac{\partial f_1}{\partial x} = x^2+y-\frac{x}{\sqrt5 \sqrt{1-x^2-3y^2}} = 0 \\[2mm]
\frac{\partial f_1}{\partial y} = x-\frac{3y}{\sqrt5 \sqrt{1-x^2-3y^2}} = 0
\end{cases},
\]
has only one solution in the interior of $E_1$, that is $(x_1,y_1)=(0.81267\ldots, 0.243532\ldots)$ such that $f_1(x_1,y_1) = 0.5566178\ldots$.

\medskip
Now, let consider the function $f_1$ on the boundary of $E_1$:
\begin{itemize}
  \item[-] $f_1(x,0)=\frac13x^3+\frac{1}{\sqrt5}\sqrt{1-x^2}\le \frac{1}{\sqrt5}=0.4472\ldots$ for $0\le x\le1$, with maximum obtained for $x=0$;
  \item[-] $f_1(0,y)=\frac{1}{\sqrt5}\sqrt{1-3y^2}\le \frac{1}{\sqrt5}=0.4472\ldots$ for $0\le y\le1/\sqrt3$, with maximum obtained for  $y=0$;
  \item[-] $f_1(1,y)=f_1(1,0)=\frac13$;
  \item[-] $f_1\left(x,\frac{1}{\sqrt3}\sqrt{1-x^2}\right) = \frac13x^3+\frac{1}{\sqrt3}x\sqrt{1-x^2}\le \frac{1}{\sqrt5}=0.4472\ldots$ for $0\le x\le1$, with maximum obtained for  $x=0.898344\ldots$.
\end{itemize}

\medskip

Summarizing the above analysis brings the conclusion that
\[|\gamma_3| \le f_1(x_1,y_1)=0.5566178\ldots.\]
\end{proof}

\medskip

\section{The second Hankel determinant for logarithmic coefficients}

\begin{thm}
Let $f\in\es$ is given by \eqref{e1}. Then
\[|H_{2,1}(F_f/2)| = \left| \gamma_1\gamma_3-\gamma_2^2 \right| \le \frac13.\]
\end{thm}

\begin{proof}
For a function $f$ from $\es$, using \eqref{loghank} and \eqref{e9}, we receive
\begin{equation}\label{eq-5n}
\gamma_1\gamma_3-\gamma_2^2 = \omega_{11}\omega_{33} + \omega_{11}^2\omega_{13} - \omega_{13}^2 - \frac14\omega_{11}^4,
\end{equation}
and from the last relation in \eqref{e9},
\[ \omega_{33} = \omega_{15} - \omega_{11}\omega_{13} +\frac13\omega_{11}^3. \]
So,
\[  \gamma_1\gamma_3-\gamma_2^2 = \omega_{11}\omega_{15} - \omega_{13}^2 + \frac{1}{12}\omega_{11}^4,  \]
and further,
\[
\begin{split}
\left| \gamma_1\gamma_3-\gamma_2^2 \right| & \le |\omega_{11}| |\omega_{15}| + |\omega_{13}|^2 + \frac{1}{12}|\omega_{11}|^4\\
& \le \frac{1}{\sqrt5} |\omega_{11}|\sqrt{1-|\omega_{11}|^2-3|\omega_{13}|^2} + |\omega_{13}|^2 + \frac{1}{12}|\omega_{11}|^4 \\
& =: f_4(|\omega_{11}|,|\omega_{13}|)
\end{split}
 \]
where
\[f_4(x,y) = \frac{1}{\sqrt5} x\sqrt{1-x^2-3y^2} + y^2 + \frac{1}{12}x^4 \]
and $(x,y) \in  E_4 \equiv \left\{ (x,y) : 0\le x\le1, 0\le y\le \frac{1}{\sqrt3}\sqrt{1-x^2}\right\}$.

Now, for the first partial derivatives of $f_4$ we have $\frac{\partial f_4}{\partial x} = \frac{x^3}{3}-\frac{x^2}{\sqrt{5} \sqrt{-x^2-3 y^2+1}}+\frac{\sqrt{1-x^2-3 y^2}}{\sqrt{5}}$ and $\frac{\partial f_4}{\partial y} = 2 y-\frac{3 x y}{\sqrt{5} \sqrt{1-x^2-3 y^2}}$. The second one being zero implies $\sqrt{1-x^2-3 y^2} = \frac{3 x}{2 \sqrt{5}}$, which brought in the first gives $\frac{\partial f_4}{\partial x} = \frac{x^3}{3}-\frac{11 x}{30}$. It is easy to check that $\frac{x^3}{3}-\frac{11 x}{30}$ is negative for $0\le x\le1$. Thus, $f_4$ has no critical points in the interior of $E_4$.

\medskip

On the boundary of $E_4$ we have
\begin{itemize}
  \item[-] $f_4(x,0)= \frac{x^4}{12}+\frac{x\sqrt{1-x^2} }{\sqrt{5}} \le f_4(0.78167\ldots, 0) = 0.2491\ldots $ for $0\le x\le1$;
  \item[-] $f_4(0,y)= y^2 \le \frac13(1-x^2)\le\frac13$;
  \item[-] $f_4(1,y)=f_4(1,0)=\frac{1}{12}$;
  \item[-] $f_4\left(x,\frac{1}{\sqrt3}\sqrt{1-x^2}\right) = \frac{1}{12} \left(2-x^2\right)^2 \le \frac13$ for $0\le x\le1$.
\end{itemize}

\medskip

Finally, we conclude that
\[\left| \gamma_1\gamma_3-\gamma_2^2 \right| \le \frac13.\]
\end{proof}

\medskip

\begin{rem}
The estimate from the previous theorem is probably not sharp since the natural way for a function $f$ from $\mathcal{S}$ to reach equality sign in the estimate is to satisfy $x=|\omega_{11}|=0$ and $y=|\omega_{13}|=\frac{1}{\sqrt3}$, leading to $a_2=0$ and $a_3=\frac{2}{\sqrt3}=1.1547\ldots$. This is in contradiction with the well known result that for functions $f\in\mathcal{S}$, $|a_3-a_2^2|\le1$, reducing in our case to $|a_3|\le1$.
\end{rem}


\vspace{5mm}

\begin{thebibliography}{99}

\bibitem{cho}
N.E. Cho, B. Kowalczyk, O.S. Kwon et al. \emph{On the third logarithmic coefficient in some subclasses of close-to-convex functions}. Revista de la Real Academia de Ciencias Exactas, F\'{i}sicas y Naturales. Serie A. Matem\'{a}ticas,  114, 52 (2020).

\bibitem{Bra85} L. De Branges, \emph{A proof of the Bieberbach conjecture}, Acta Math. {\bf 154} (1985), no. 1--2, 137--152.


\bibitem{duren}
P.L. Duren, Univalent function, Springer-Verlag, New York, 1983.



\bibitem{hayman-68}
W.K. Hayman, On the second Hankel determinant of mean univalent functions, \textit{Proc. London Math. Soc.} \textbf{3}(18) (1968), 77-–94.



\bibitem{janteng-06}
A. Janteng, S.A. Halim, M. Darus, Coefficient inequality for a function whose derivative has a positive real part. \textit{J. Inequal. Pure Appl. Math.} \textbf{7}(2) (2006), Article 50, 5 pp.

\bibitem{janteng-07}
A. Janteng,  S.A. Halim, M. Darus, Hankel determinant for starlike and convex functions, \textit{Int. J. Math. Anal. (Ruse).} \textbf{1}(13-16) (2007), 619–-625.

\bibitem{Kowalczyk-22}
B. Kowalczyk,  A. Lecko, Second Hankel determinant of logarithmic coefficients of convex and starlike functions, \textit{Bull. Aust. Math. Soc.} \textbf{105}(3) (2022), 458--467.


\bibitem{Kowalczyk-23}
B. Kowalczyk, A. Lecko, The second Hankel determinant of the logarithmic coefficients of strongly starlike and strongly convex functions, {\it Rev. Real Acad. Cienc. Exactas Fis. Nat. Ser. A-Mat.},  {\bf117}, 91 (2023).

\bibitem{Lebedev}
N.A. Lebedev, Area principle in the theory of univalent functions, Published by "Nauka", Moscow, 1975 (in Russian).

\bibitem{OT-Turk}
M. Obradovi\`{c}, N. Tuneski, The third logarithmic coefficient for the class S, {\it Turkish Journal of Mathematics}, {\bf44} (2020), 1950--1954.


\bibitem{OB-NT-n2}
M. Obradovi\'{c}, N. Tuneski, Zalcman and generalized Zalcman conjecture for a subclass of univalent functions, {\it Novi Sad J. of Mathematics}, {\bf52}(1) (2022), 185--190.


\bibitem{ravi}
V. Ravichandran, S. Verma,  Generalized Zalcman conjecture for some classes of analytic functions, {\it  J. Math. Anal. Appl.} {\bf450}(1) (2017), 592--605.


\bibitem{DTV-book}
D.K. Thomas, N. Tuneski, A. Vasudevarao, Univalent Functions: A Primer, \emph{De Gruyter Studies in Mathematics} {\bf 69}, De Gruyter, Berlin, Boston, 2018.



\end{thebibliography}
\end{document}